\documentclass[12pt]{article}

\usepackage{graphicx}
\usepackage{nameref}
\usepackage[shortlabels,inline]{enumitem}
\usepackage{empheq}
\usepackage[shortlabels]{enumitem}
\setlist[enumerate]{nosep}
\usepackage[doc]{optional}
\usepackage{xcolor}
\usepackage[colorlinks=true,
            linkcolor=refkey,
            urlcolor=lblue,
            citecolor=red]{hyperref}
\usepackage{float}
\usepackage{soul}
\usepackage{graphicx}
\definecolor{labelkey}{rgb}{0,0.08,0.45}
\definecolor{refkey}{rgb}{0,0.6,0.0}
\definecolor{Brown}{rgb}{0.45,0.0,0.05}
\definecolor{lime}{rgb}{0.00,0.8,0.0}
\definecolor{lblue}{rgb}{0.5,0.5,0.99}

 \usepackage{mathpazo}

\colorlet{hlcyan}{cyan!30}

\usepackage{stmaryrd}
\usepackage{amssymb}
\makeatletter
\def\namedlabel#1#2{\begingroup
   \def\@currentlabel{#2}%
   \label{#1}\endgroup
}
\makeatother

\newcommand{\seppthree}{\setlength{\itemsep}{-3pt}}

\newcommand{\To}{\ensuremath{\rightrightarrows}}

\newcommand{\fenv}[1]%
{\ensuremath{\,\overrightarrow{\operatorname{env}}_{#1}}}
\newcommand{\benv}[1]%
{\ensuremath{\,\overleftarrow{\operatorname{env}}_{#1}}}

\newcommand{\scal}[2]{\left\langle{#1},{#2}  \right\rangle}

\newcommand{\RR}{\ensuremath{\mathbb R}}

\newcommand{\Id}{\ensuremath{\operatorname{Id}}}

{\begin{list}{}{%
\settowidth{\labelwidth}{\textrm{#1~}}%
\setlength{\leftmargin}{\labelwidth+\labelsep}}}
{\end{list}}
\usepackage{amsthm}
\usepackage[capitalize,nameinlink]{cleveref}
\crefname{equation}{}{equations}
\crefname{chapter}{Appendix}{chapters}
\crefname{item}{}{items}
\crefname{enumi}{}{}
\newtheorem{theorem}{Theorem}[section]
\newtheorem{lemma}[theorem]{Lemma}

\newtheorem{proposition}[theorem]{Proposition}

\newtheorem{example}[theorem]{Example}
\newtheorem{ex}[theorem]{Example}
\newtheorem{fact}[theorem]{Fact}

\providecommand{\RR}{\mathbb{R}}

\providecommand{\Id}{\operatorname{{ Id}}}

\providecommand{\To}{\rightrightarrows}

\providecommand{\Id}{\operatorname{Id}}
\providecommand{\PR}{\operatorname{P}}

\providecommand{\RR}{\mathbb{R}}

\definecolor{myblue}{rgb}{.8, .8, 1}

\allowdisplaybreaks 

\begin{document}

%

\title{\textsc{More on Arag\'on Artacho–Campoy algorithm Operators}}
\author{
Salihah Thabet Alwadani\thanks{
Mathematics, Yanbu Industrial College, The Royal Comission for Jubail and Yanbu, 30436,
 Saudi Arabia. E-mail:
\texttt{salihah.s.alwadani@gmail.com}.}
}

\date{March 7, 2024 }
\maketitle

\vskip 8mm

\begin{abstract} 
The Arag\'on Artacho–Campoy algorithm (AACA) is a new method for finding zeros of sums of monotone operators. In this paper we complete the analysis of their algorithm by defining their operator using Douglas Rachford operator and then study the effects of the order of the two possible Arag\'on Artacho–Campoy operators.
\end{abstract}

{\small

}

\maketitle

\section{Introduction}

Throughout, we assume that 
\begin{equation}
\text{$X$ is
a real Hilbert space with inner product
$\scal{\cdot}{\cdot}\colon X\times X\to\RR$, }
\end{equation}
and induced norm $\|\cdot\|\colon X\to\RR\colon x\mapsto \sqrt{\scal{x}{x}}$. We also assume that $A: X \To X$ and $B: X \To X$ are maximally monotone operators. For more details about maximally monotone operators we refre the reader to \cite{auslender2000asymptotic}, \cite{BC2017},  \cite{aborwein2010fifty}, \cite{borwein2010convex}, \cite{borwein2004techniques}, \cite{burachik2002maximal}, \cite{combettes2018monotone}, \cite{eckstein1989splitting}, and the references therein. In \cite{auslender2000asymptotic},  Auslender and Teboulle provide essential tools used to study monotone graph. They focuse on the behavior of a given subset of $\RR^{n}$ at infinity. By using real analysis and geometric concepts, they come up with mathematical treatment to study the asymptotic behaviour of sets. Moreover, Heinz and Patrick book \cite{BC2017} is one of the best source to learn about non-linear analysis, namely, convex analysis, monotone operators, and fixed point theory of operators.  Additionally, \cite{aborwein2010fifty} shows the importance of the maximal monotone operators and describes the progres which has been made in the world of monotone operators during the past decade. Furthermore, \cite{borwein2010convex} gives a survey that shows the developments on the theory behind the monotone operators. It is well known that a promintet example of maximal monotone operators is the subdifferential operators, which were investigated in section 5.1.6 in \cite{borwein2004techniques} . \\

What is more Burachik and Svaiter come up with new connections between maximal monotone operators and convex functions.  They show that there is a family of convex functions associated to each maximal monotone operator.  They study the family of convex functions, and determine its extremal elements by using the idae of family of convex functions, see \cite{burachik2002maximal}. After that, Patrick in \cite{combettes2018monotone} reviews the properities of subdifferentals operators as maximally montone operators and studies those of proximity operators as resolvants of maximally monotone operators.  In \cite{eckstein1989splitting} a cohesive treatment of monotone setvalued operators using mathematical programming is studied with details.

The \emph{resolvent} and the \emph{reflected resolvent} associated with $A$ are
\begin{equation}\label{quli1}
J_{A} = (\Id + A)^{-1} \ \ \text{and} \ \ R_{A} = 2 J_{A} - \Id,
\end{equation}
 respectively. Suppose that
\begin{equation}
A \ \text{and} \ B \ \text{are maximally monotone on} \ X, \ w \in X, \ \text{and} \ \gamma \in  \left]0,1 \right[.
\end{equation}
\begin{fact}\normalfont
The resolvent averages between $A$, $B$ and $N_{w}$ are 
\begin{equation}
A_{\gamma}: \mathcal{H} \To \mathcal{H}: x \mapsto A \Big( \gamma^{-1} \big( x - (1 - \gamma ) w \big) \Big)+  \gamma^{-1} \big( 1 - \gamma \big) \big( x - w\big),
\end{equation}
and 
\begin{equation}
B_{\gamma}: \mathcal{H} \To \mathcal{H}: x \mapsto B \Big( \gamma^{-1} \big( x - (1 - \gamma ) w \big) \Big)+  \gamma^{-1} \big( 1 - \gamma \big) \big( x - w\big).
\end{equation}	
\end{fact}

\begin{fact}\normalfont
$A_{\gamma}$ and $B_{\gamma}$ are maximally monotone and their resolvents are given by
\begin{equation}\label{eq:ros}
J_{A_\gamma } = \gamma J_{A} + \big( 1 - \gamma \big) w \ \ \text{and} \ \ J_{B_\gamma } = \gamma J_{B} + \big( 1 - \gamma \big) w,
\end{equation}
respectively. Moreover, reflected resolvents are 
\begin{equation}\label{eq:ref}
R_{A_\gamma} = 2 \gamma J_{A} + 2 \big( 1 - \gamma \big) w - \Id, \ \ \text{and} \ \ R_{B_\gamma} = 2 \gamma J_{B} + 2 \big( 1 - \gamma \big) w - \Id,
\end{equation}
 respectively. Then Arag\'on Artacho–Campoy operator \cite{alwadani2018asymptotic} associated with the ordered pair of operators $\big( A_\gamma, B_\gamma \big)$ is
 \begin{equation}\label{twithouttide}
 T_{A_\gamma, B_\gamma } = \big( 1 - \lambda \big) \Id + \lambda R_{B_\gamma} R_{A_\gamma}.
 \end{equation}
\end{fact}

\begin{fact}\label{fact:1}\normalfont The Douglas–Rachford splitting operator \cite{lions1979splitting} associated with the ordered pair of operators $\big( A, B\big)$ is
\begin{equation}\label{toon}
T_{A,B} = \frac{1}{2} \big( \Id + R_{B} R_{A}\big) = \Id - J_{A} + J_{B} R_{A}.
\end{equation}	
\end{fact}

 By doing simple calculations, we can find that 
 \begin{equation}\label{cccs}
 T_{B, A} = \Id + J_{A} R_{B} - J_{B}.
 \end{equation}
 In this paper, We  study the connection between the Arag\'on Artacho–Campoy operators $T_{A_\gamma, B_\gamma}$ and $T_{B_\gamma, A_\gamma}$. We summarize the main results in this paper as follows:
 \begin{itemize}
 \item Important properities about $J_{A_\gamma}$ and $A_\gamma$ are given in \cref{Prop11}. These properities will be useful for the study.
 \item We give formulas for the  Arag\'on Artacho–Campoy operators using the  Douglas–Rachford splitting operator ( see \cref{Lem1} ). For more information about Douglas–Rachford splitting Algorithm see \cite{fu2020anderson}, \cite{bauschke2016range}, \cite{bauschke2023douglas}, \cite{censor2016new}, \cite{eckstein1992douglas}, and \cite{lindstrom2021survey}. \\  \cite{bauschke2016range}, \cite{bauschke2023douglas},  \cite{censor2016new}, and \cite{lindstrom2021survey} help to undersytand more about the behaviour of DRS. Paper \cite{bauschke2016range} studies the range of the DRS systematically. Under the assumption that the operators are $3^{*}$ monotone operators. While the second one helps to understand the behavior of the shadow sequence when the given functions have disjoint domains. The main result of this paper is proving the weak and value convergence of the shadow sequence generated by the Douglas–Rachford algorithm. Paper \cite{censor2016new} aims to solve convex feasibilty problems by using new new algorithmic structures with DRS operators. Paper \cite{lindstrom2021survey} gives comprehensive survey about the developments of the DRS methods. Additionally, \cite{fu2020anderson} shows an amazing connection between the alternating direction multiplier method (ADMM) and Douglas Rachford Splitting method (DRS) for convex problem. Finally, paper \cite{eckstein1992douglas} shows that the proximal point algorithm encompasses the DRS method as a specific instance, which is employed for locsting a zero of the combined sum of two monotone operators. 
 \item With the assumption $A$ is affine relation, we prove that 
 $   R_{A_\gamma} T^{n}_{A_\gamma, B_\gamma } = T^{n}_{B_\gamma, A_\gamma } R_{A_\gamma} $ (see \cref{Lem:rs3}).
 \item We illustrate the results by giving two examples (see \cref{rea1} and \cref{EEXA1}).
 \item  We proved that the equality will not hold by replacing $A_\gamma$ by $B_\gamma$ in the privous result ( see \cref{EEXA1} \ref{Ex.17}, \ref{Ex.18}, and \ref{Ex.19} ).
 \end{itemize}  
 The notation used in the paper is standard and follows largely, e.g., \cite{alwadani2021behaviour}, \cite{alwadani2018asymptotic}, and \cite{BC2017}.

\section{New Results}
\begin{proposition}\label{Prop11}\normalfont Let $\gamma \in \left] 0, 1 \right[$ and suppose that $A$ is an affine relation. Then the following hold;
\begin{enumerate}
\item\label{re:1} $J_{A_\gamma}$ is affine.
\item\label{re:2} $A_\gamma$ is an affine relation.	
\end{enumerate}	
\end{proposition}
\begin{proof}
\ref{re:1}: It follows from \cite[Lemma 2.3]{bauschke2016order} or \cite[Theorem  2.1(xix)]{bauschke2012firmly} that $J_{A}$ is affine. By using \eqref{eq:ros}, we obtain $J_{A_\gamma}$ is affine. Then, we have 
$J_{A_\gamma}$ is affine. \\
\ref{re:2}: From \ref{re:1}
 $J_{A_\gamma}$ is affine $\Leftrightarrow \big( \Id + A_\gamma\big)^{-1} $ is an affine relation $\Leftrightarrow \big( \Id + A_\gamma \big)$ is an affine relation $\Leftrightarrow A_\gamma$ is an affine relation.
\end{proof}
\begin{lemma}\label{Lem1}\normalfont Let $\gamma \in \left] 0, 1 \right[$, $\lambda \in \left] 0, 1 \right]$. We obtain 
\begin{align}
R_{B_\gamma} R_{A_\gamma} & = \Id + 2 J_{B_\gamma} R_{A_\gamma} - 2  J_{A_\gamma}\label{eq:m1}\\
& = \Id + 2 \gamma J_{B} R_{A_\gamma} - 2 \gamma J_{A}\label{eq:m2} \\
& = T_{A,B} + (1 - 2 \gamma) J_{A} - J_{B} R_{A}+ 2 \gamma J_{B} R_{A_\gamma}.\label{eq:m3}
\end{align}	
and 
\begin{align}
R_{A_\gamma} R_{B_\gamma} & = \Id + 2 J_{A_\gamma} R_{B_\gamma} - 2  J_{B_\gamma }\\
& = \Id + 2 \gamma J_{A} R_{B_\gamma} - 2 \gamma J_{B} \\
& = T_{B,A} + (1 - 2 \gamma) J_{B} - J_{A} R_{B} + 2 \gamma J_{A} R_{B_\gamma}.
\end{align}	
Moreover, 
\begin{align}
T_{A_\gamma, B_\gamma} & = \Id + 2 \lambda J_{B_\gamma} R_{A_\gamma} - 2 \lambda J_{A_\gamma}\label{eq:m4} \\
& = \Id + 2 \lambda \gamma J_{B} R_{A_\gamma} - 2 \lambda \gamma J_{A}\label{eq:m5}\\
& = \Id + 2 \lambda \gamma \big( J_{B} R_{A_\gamma} -  J_{A} \big) \label{eq:m566}\\
& =T_{A,B} + (1 - 2 \lambda \gamma ) J_{A} - J_{B} R_{A} + 2 \lambda \gamma J_{B} R_{A_\gamma} \label{eq:m5676} \\
& = T_{B, A} + J_{B} - J_{A} R_{B} + 2 \lambda \gamma J_{B} R_{A_\gamma } - 2 \lambda \gamma J_{A}.\label{eq:m567s6}
\end{align}
and 
\begin{align}
T_{B_\gamma, A_\gamma} & = \Id + 2 \lambda J_{A_\gamma} R_{B_\gamma} - 2 \lambda J_{B_\gamma} \\
& = \Id + 2 \lambda \gamma J_{A} R_{B_\gamma} - 2 \lambda \gamma J_{B}\\
& = \Id + 2 \lambda \gamma \big( J_{A} R_{B_\gamma} - J_{B} \big) \\
& = T_{B,A} + (1 - 2 \lambda \gamma ) J_{B} - J_{A} R_{B} + 2 \lambda \gamma J_{A} R_{B_\gamma}.
\end{align}
\end{lemma}

\begin{proof}
From \eqref{eq:ref}, we obtain 
\begin{align*}
R_{B_\gamma} R_{A_\gamma } & = \big( 2 J_{B_\gamma} - \Id \big) R_{A_\gamma } \\
& = 2 J_{B_\gamma} R_{A_\gamma} - R_{A_\gamma} \\
& = 2 J_{B_\gamma} R_{A_\gamma} - \big( 2 J_{A_\gamma} - \Id\big) \\
& = \Id +  2 J_{B_\gamma} R_{A_\gamma} - 2 J_{A_\gamma},
\end{align*}
which prove \eqref{eq:m1}. From \eqref{eq:ros} and  \eqref{eq:m1} we have 
\begin{align*}
R_{B_\gamma} R_{A_\gamma } & = \Id + 2 \big( \gamma J_{B} + (1 - \gamma) w \big) R_{A_\gamma} - 2 \big( \gamma J_{A} + (1 - \gamma ) w \big) \\
& = \Id + 2 \gamma J_{B} R_{A_\gamma} + 2 ( 1 - \gamma ) w - 2 \gamma J_{A} - 2 (1 - \gamma ) w \\
& = \Id + 2 \gamma J_{B} R_{A_\gamma} - 2 \gamma J_{A},
\end{align*}
which verifies \eqref{eq:m2}. Next, from Fact~\ref{fact:1} we have 
\begin{align*} 
R_{B_\gamma} R_{A_\gamma} & = T_{A,B} + J_{A} - J_{B} R_{A} + 2 \gamma J_{B} R_{A_\gamma} - 2 \gamma J_{A} \\
& = T_{A,B} + (1 - 2 \gamma ) J_{A} - J_{B} R_{A} + 2 \gamma J_{B} R_{A_\gamma}.
\end{align*}
This verifies \eqref{eq:m3}. The proof of $R_{A_\gamma}R_{B_\gamma}$ is similar to $R_{B_\gamma}R_{A_\gamma}$. From \eqref{twithouttide}, we have 
\begin{align*}
T_{A_\gamma, B_\gamma} & = ( 1 - \lambda) \Id + \lambda R_{B_\gamma }R_{A_\gamma } \\
& = ( 1 - \lambda) \Id + \lambda \big(  \Id + 2 J_{B_\gamma} R_{A_\gamma} - 2  J_{A_\gamma} \big)  \ \ (\text{from} \eqref{eq:m1})\\
& = \Id + 2 \lambda J_{B_\gamma} R_{A_\gamma} - 2 \lambda J_{A_\gamma},
\end{align*}
which verifies \eqref{eq:m4}. By using \eqref{eq:ros} and \eqref{eq:m4}, we obtain
\begin{align*} 
T_{A_\gamma, B_\gamma} &  = \Id + 2 \lambda \big( \gamma J_{B} + (1 - \gamma ) w\big) R_{A_\gamma} - 2 \lambda \big( \gamma J_{A} + (1 - \gamma ) w \big) \\
& = \Id + 2 \lambda \gamma J_{B} R_{A_\gamma } + 2 \lambda (1 - \gamma ) w - 2 \lambda \gamma J_{A} - 2 \lambda (1 - \gamma ) w \\
& =  \Id + 2 \lambda \gamma J_{B} R_{A_\gamma} - 2 \lambda \gamma J_{A}. \\
& = \Id + 2 \lambda \gamma \big( J_{B} R_{A_\gamma} - J_{A}\big)
\end{align*}
This verifies \cref{eq:m5} and \cref{eq:m566}. Finally from \cref{eq:m5} and \cref{toon}, we obtain 
\begin{align*}
T_{A_\gamma, B_\gamma} &  =  T_{A,B} + J_{A}- J_{B} R_{A} + 2 \lambda \gamma J_{B} R_{A_\gamma } - 2 \lambda \gamma J_{A} \\
& =  T_{A,B} + (1 - 2 \lambda \gamma ) J_{A} - J_{B} R_{A} + 2 \lambda \gamma J_{B} R_{A_\gamma}.
\end{align*}
Combining \cref{cccs} and \cref{eq:m5} to get \cref{eq:m567s6}.
The proof of $T_{B_\gamma, A_\gamma}$ is similar to $T_{A_\gamma, B_\gamma}$.
\end{proof}

\begin{example}\normalfont\label{rea1} Let $w \in \mathcal{H}$, $U$ be a closed linear subspace, $\gamma \in \left] 0, 1 \right[$, and $\lambda \in \left] 0, 1 \right]$. Suppose that $A= \Id -v$, where $v \in U^\perp$ and $B= \PR_{ a+U}$, where $a \in \mathcal{H}$. Recall from \eqref{toon} that  $$T_{A,B} = \Id - J_{A}+ J_{B} R_{A},$$  and from \eqref{twithouttide}  $$T_{(A_\gamma, B_\gamma )} = ( 1 - \lambda ) \Id + \lambda \big( R_{B_\gamma} R_{A_\gamma }\big).$$
Then the following hold:
\begin{enumerate} 
\item\label{eqq:3} $J_{A} = \big(( \Id + v)/2\big) $ and $R_{A} = v$. 
\item\label{eqq:4} $J_{A_\gamma} =  \gamma \big(( \Id + v) / 2\big)+ ( 1 - \gamma ) w$. Moreover,
$$ R_{A_\gamma} = \gamma v - (1 - \gamma) \Id + 2 (1 - \gamma) w. $$ 
\item\label{eqq:1} $J_{B} = \big( \Id - \frac{1}{2} \PR_{ U}\big) - \PR_{ U^{\perp}} a \ \ $ and $\ \ R_{B} = \big( \Id - \PR_{ U}\big) - 2 \PR_{ U^{\perp}} a$.
\item\label{eqq:2} We have 
$$J_{B_\gamma } = \gamma \Big( \big( \Id - \frac{1}{2} \PR_{ U}\big) - \PR_{ U^{\perp}} a\Big) + (1 - \gamma ) w,$$ and
$$R_{B_\gamma} =  (2 \gamma - 1) \Id - \gamma \PR_{ U} - 2 \gamma \PR_{ U^{\perp}} a + 2 (1 - \gamma ) w.$$ 
\item\label{eq:s1} $T_{A,B}=  \big( (\Id + v)/2 \big) -  \PR_{ U^{\perp}} a$.
\item\label{eq:s2} $T_{B, A} = \big( (\Id + v) / 2 \big)$.
\item\label{ken:01} $J_{B} R_{A} = v - \PR_{ U^\perp} a$.
\item\label{ken:011} $J_{A} R_{B} =  \big( (\Id + v) / 2 \big) - ( \PR_{ U} / 2) - \PR_{ U^{\perp}} a$.
\item\label{ken:11} $J_{B} R_{A_\gamma} = \gamma v + (1 - \gamma) \Big( \Big( \frac{1}{2} \PR_{ U} - \Id\big) - \big( \PR_{ U} - 2 \Id \big) w \Big) - \PR_{ U^{\perp}} a$.
\item\label{ken:0011} $J_{A} R_{B_\gamma } = \frac{1}{2} \Big( (2 \gamma - 1) \Id - \gamma \PR_{ U} - 2 \gamma \PR_{ U^\perp} a + 2 (1 - \gamma ) w + v\Big)$.
\item\label{tab:1} Suppose $k := \lambda \gamma \Big[ (2 \gamma - 1)v + 4 (1 - \gamma ) w -2 ( 1 - \gamma ) \PR_{ U} w - 2 \PR_{ U^{\perp}} a\Big]$. Then 
\begin{align}
T_{A_\gamma, B_\gamma} (x) & = \big( 1  -  \lambda \gamma (3 - 2 \gamma) \big) x + \lambda \gamma ( 1 -  \gamma ) \PR_{ U} x + k. 
\end{align}	
\item\label{tab:2} Suppose $ l:= \lambda \gamma \Big[ 2 (1 - \gamma ) \PR_{ U^\perp} a + v + 2 ( 1 - \gamma ) w\Big]$. Then 
\begin{equation}
T_{B_\gamma, A_\gamma } (x) = \Big( 1 - \lambda \gamma (3 - 2 \gamma )\Big) x + \lambda \gamma (1 - \gamma ) \PR_{ U} x + l.
\end{equation}
\end{enumerate}
\end{example}

\begin{proof}
\ref{eqq:3}: Let $y \in \mathcal{H}$ and set $x = J_{A} y $. Then $y \in ( \Id +A) x \Leftrightarrow y =  2x -v  \Leftrightarrow x = ((y+ v)/ 2) \Rightarrow J_{A} = \big( (\Id+ v)/2 \big)$. Therefore, $ R_{A}  =  2  \big( (\Id+ v)/2 \big) - \Id  \Leftrightarrow R_{B} = v $ by \eqref{quli1}. \\
\ref{eqq:4}: Combine \ref{eqq:3} and \eqref{eq:ros} gives 
\begin{align*}
R_{A_\gamma } (x)& = 2 \gamma  \big( (x+ v)/2 \big) + 2 (1 - \gamma) w- x  \\
& = \gamma \big( x + v\big) + 2 \big( 1 - \gamma \big) w - x \\
& = \gamma v - \big( 1 - \gamma \big) x + 2 \big( 1 - \gamma \big) w.
\end{align*}	
\ref{eqq:1}: Let $y \in \mathcal{H}$ and set $x = J_{B} y $. Our goal is to find $x$. We have
\begin{equation*}
\begin{split}
y \in ( \Id + \PR_{a +U} ) x & \Leftrightarrow y =  x +  a + \PR_{U} (x - a)\\
& \Leftrightarrow y =  x +  (\Id -  \PR_{U}) a+ \PR_{U} x  \\
& \Leftrightarrow y =  x +  \PR_{ U^{\perp}} a+ \PR_{U} x .
\end{split}
\end{equation*}
Hence,
\begin{equation}\label{ranse1yy2707}
y = x +  \PR_{ U^{\perp}} a+ x^{*} ,  \ \textrm{where} \ \  x^{*} = \PR_{U} x.
\end{equation} 
Applying $\PR_{U}$ to \eqref{ranse1yy2707} gives
\begin{equation}\label{ranse1yy2708}
\PR_{U}y = \PR_{U}x+ \PR_{U} \PR_{ U^{\perp}} a + x^{*}  \Leftrightarrow \PR_{U}y  = 2 x^{*}   \Leftrightarrow x^{*} = \frac{1}{2} \PR_{U} y. 
\end{equation}
Substituting \eqref{ranse1yy2708} back into \eqref{ranse1yy2707} gives 
\begin{equation*}\label{ranse1yy2709}
y = x + \PR_{U^{\perp}}a + \frac{1}{2} \PR_{U} y  \Leftrightarrow x = \Big( \Id - \frac{1}{2} \PR_{U} \Big) y -  \PR_{U^{\perp}} a. 
\end{equation*}
Therefore,
\begin{equation*}\label{ranse1yy2710}
J_{B} = (\Id - \frac{1}{2} \PR_{U}) -  \PR_{U^{\perp}} a, 
\end{equation*}
and 	
\begin{equation*}\label{ranse1yy27201}
\begin{split}
R_{B}  & = 2 \ \big( \Id - \frac{1}{2} \PR_{U} \big) - 2 \PR_{U^{\perp}} a  - \Id  =  ( \Id -  \PR_{U}) - 2 \PR_{U^{\perp}} a,
\end{split}
\end{equation*}
by \eqref{quli1}.\\
\ref{eqq:2}: It follows from \eqref{eq:ros} and \ref{eqq:1} that 
\begin{align*}
J_{B_\gamma} &= \gamma J_{B}+ 2 (1 - \gamma) w \\
& = \gamma \Big(  \Big(\Id - \frac{1}{2} \PR_{U} \Big) -  \PR_{U^{\perp}} a \Big) + (1 - \gamma) w.
\end{align*}
By \eqref{eq:ref}, we have 
\begin{align*}
R_{B_\gamma } & = 2 \gamma \Big(  \Big(\Id - \frac{1}{2} \PR_{U} \Big) -  \PR_{U^{\perp}} a \Big) + 2 (1 - \gamma) w - \Id \\
& = (2 \gamma - 1) \Id - \gamma \PR_{ U} - 2 \gamma \PR_{ U^{\perp}} a + 2 (1 - \gamma ) w.
\end{align*}
\ref{eq:s1}: Using \ref{eqq:1}, \ref{eqq:3}, and \eqref{toon} give 
\begin{align*}
T_{A,B} ( x ) & =  x - \Big( \frac{x + v}{2}\Big) + \Big( \Id - \frac{1}{2} \PR_{ U} - \PR_{ U^{\perp}} a\Big) R_{A}x \\
& = \frac{x}{2} - \frac{v}{2} + R_{A} x - \frac{1}{2} \PR_{ U} R_{A} x- \PR_{ U^{\perp}} a \\
& = \Big( \frac{ x+ v}{2}\Big) - \PR_{ U^{\perp}} a.
\end{align*}
\ref{eq:s2}: From \ref{eqq:1}, \ref{eqq:4}, and \eqref{toon}, we have 
\begin{align*}
T_{B, A} (x)& = \frac{1}{2} x + \frac{1}{2} R_{A} R_{B} x \\
& = \frac{1}{2} x + \frac{1}{2} (v) \Big( \big( x - \PR_{ U} x \big) - 2 \PR_{ U^{\perp}} a \Big) \\
& = \frac{1}{2} \big( x + v\big).
\end{align*}

\ref{ken:01}: By using \ref{eqq:1} and \ref{eqq:3}, we obtain 
  \begin{align*}
   J_{B} R_{A}x&= \big( \Id - \frac{1}{2} \PR_{ U} - \PR_{ U^{\perp}} a \big) R_{A}x \\
   & = R_{A} x - \frac{1}{2} \PR_{ U} R_{A}x - \PR_{ U^{\perp}} a \\
   & = v - \PR_{ U^{\perp}} a.
  \end{align*}
\ref{ken:011}: Using \ref{eqq:3} and \ref{eqq:1} gives 
\begin{align*} 
 J_{A} R_{B}x&= \Big( \frac{\Id + v}{2}\Big) R_{B} x \\
 & = \frac{1}{2} R_{B} x + \frac{1}{2} v \\
 & = \frac{1}{2} \Big( x - \PR_{ U} x - 2 \PR_{ U^\perp} a\Big) + \frac{1}{2} v \\
 & =   \Big( \frac{x + v}{2}\Big) - \frac{1}{2} \PR_{ U} x - \PR_{ U^\perp} a.
\end{align*}

\ref{ken:11}: By using \ref{eqq:4} and \ref{eqq:1}, we obtain 
\begin{align*}  
	J_{B} R_{A_\gamma}x&= \big( \Id - \frac{1}{2} \PR_{ U} - \PR_{ U^{\perp}} a \big) R_{A_\gamma}x \\
	& = R_{A_\gamma} x - \frac{1}{2} \PR_{ U} R_{A_\gamma}x - \PR_{ U^{\perp}} a \\
	& = \gamma v - (1 - \gamma) x + 2 (1 - \gamma) w -\frac{1}{2} \PR_{ U} \Big( 2 (1 - \gamma) w - (1 - \gamma) x \Big) - \PR_{ U^{\perp}} a \\
	& = \gamma v + (1 - \gamma) \Big( \Big( \frac{1}{2} \PR_{ U} - \Id \Big) x - \big( \PR_{ U} - 2 \Id \big) w\Big)- \PR_{ U^\perp} a.
\end{align*}
\ref{ken:0011}: From \cref{eqq:3} and \cref{eqq:2}, we obtain 
\begin{align*}
J_{A} R_{B_\gamma } x & = \Big( \frac{\Id + v}{2}\Big) \Big( ( 2 \gamma -1) x - \gamma \PR_{ U} x - 2 \gamma \PR_{ U^\perp} a + 2 (1 - \gamma ) w\Big) \\
& = \frac{1}{2} \Big(  (2 \gamma - 1) x - \gamma \PR_{ U} x - 2 \gamma \PR_{ U^{\perp}} a + 2 (1 - \gamma ) w + v\Big).
\end{align*}
\ref{tab:1}: Combining \ref{ken:01}, \ref{ken:11}, \ref{eq:s1}, and \cref{eq:m5676} gives
\begin{align*} 
T_{A_\gamma, B_\gamma}x &= T_{A, B} + (1 - 2 \lambda \gamma ) J_{A} - J_{B} R_{A} + 2 \lambda \gamma J_{B} R_{A_\gamma }\\
& = \Big( \frac{x+ v}{2}\Big) - \PR_{ U^{\perp}} a + ( 1 - 2 \lambda \gamma ) \Big( \frac{x+ v}{2}\Big) - J_{B} R_{A} + 2 \lambda \gamma J_{B} R_{A_\gamma}\\
& = (x+ v) - \PR_{ U^\perp} a - \lambda \gamma (x+ v) 	- J_{B} R_{A} + 2 \lambda \gamma J_{B} R_{A_\gamma }\\
& =  (x+ v) - \lambda \gamma (x+ v) - v + 2 \lambda \gamma J_{B} R_{A_\gamma }\\
& = x - \lambda \gamma (x+ v) + 2 \lambda \gamma J_{B} R_{A_\gamma }\\
& =  x - \lambda \gamma (x+ v) + 2 \lambda \gamma \Big[  \gamma v + \frac{(1 - \gamma )}{2} \PR_{ U} x - (1 - \gamma ) (x + \PR_{ U} w - 2 w) - \PR_{ U^{\perp}} a  \Big] \\
& =  \big( 1  -  \lambda \gamma (3 - 2 \gamma) \big) x + \lambda \gamma ( 1 -  \gamma ) \PR_{ U} x + \lambda \gamma \Big( (2 \gamma - 1)v \\
& \ \ \ \ \ \ \ \ \ \ \ \ \ \ + 4 (1 - \gamma ) w -2 ( 1 - \gamma ) \PR_{ U} w - 2 \PR_{ U^{\perp}} a\Big).
\end{align*}
   
\ref{tab:2}: From \cref{eqq:1},  \cref{eq:s2},  \cref{ken:011}, and \cref{ken:0011}, we obtain 
\begin{align*}
T_{B_\gamma, A_\gamma } x &= T_{B, A} x + (1 - 2 \lambda \gamma) J_{B} x - J_{A} R_{B} x + 2 \lambda \gamma J_{A} R_{B_\gamma } x \\
& = \Big( \frac{ x+ v}{2}\Big) + (1 - 2 \lambda \gamma) J_{B} x -  \Big( \frac{ x+ v}{2}\Big) + \frac{1}{2} \PR_{ U} x + \PR_{ U^\perp} a + 2 \lambda \gamma J_{A} R_{B_\gamma } x \\
& =  (1 - 2 \lambda \gamma) \Big( x - \frac{1}{2} \PR_{ U} x - \PR_{ U^\perp}a \Big) + \frac{1}{2} \PR_{ U} x + \PR_{ U^\perp} a + 2 \lambda \gamma J_{A} R_{B_\gamma } x \\
& = ( 1 - 2 \lambda \gamma ) x + 2 \lambda \gamma \PR_{ U^\perp} a + \lambda \gamma \PR_{ U} x + 2 \lambda \gamma J_{A} R_{B_\gamma } x \\ 
& = \Big( 1 - \lambda \gamma (3 - 2 \gamma )\Big) x + \lambda \gamma (1 - \gamma ) \PR_{ U} x + \lambda \gamma \Big[ 2 (1 - \gamma ) \PR_{ U^\perp} a + v + 2 ( 1 - \gamma ) w\Big], 
\end{align*}
which verifies \ref{tab:2}.	
\end{proof}

\begin{lemma}\label{Lem:rs1}\normalfont  Let $A: \mathcal{H} \To \mathcal{H}$ be maximally monotone and $\gamma \in \left] 0, 1 \right[$. Suppose $J_{A}$ is affine. Then 
\begin{equation}
J_{A_\gamma } R_{A_\gamma } =  R_{A_\gamma } J_{A_\gamma}.
\end{equation}
\end{lemma}
\begin{proof} Combine \cref{Prop11}\ref{re:1} and  \cite[Lemma~2.4 (i)]{bauschke2016order}.
\end{proof}
	
\begin{lemma}\label{Lem:rs2}\normalfont Suppose that $A$ is an affine relation. Then we have 
\begin{align}
R_{A_\gamma} T_{A_\gamma, B_\gamma}- T_{B_\gamma, A_\gamma} R_{A_\gamma} & = 2 \big( J_{A_\gamma } T_{A_\gamma, B_\gamma} - (1 - \lambda) J_{A_\gamma} - \lambda J_{A_\gamma } R_{B_\gamma } R_{A_\gamma}\big).\label{eq:m8} \\
& = 2 \gamma \big( J_{A} T_{A,B} - (1 - \lambda) J_{A} - \lambda J_{A}  R_{B_\gamma } R_{A_\gamma}\big).\label{eq:m9}
\end{align}
\end{lemma}

\begin{proof}
It follws from \cref{Prop11}\ref{re:2} that $A_\gamma $ is an affine relation. Hence, by using \cref{eq:ref} and \cref{eq:m4}, we obtain 
\begin{align*}
R_{A_\gamma} T_{A_\gamma, B_\gamma} - T_{B_\gamma, A_\gamma} R_{A_\gamma} & = \big( 2 J_{A_\gamma} - \Id \big) T_{A_\gamma, B_\gamma} - T_{B_\gamma, A_\gamma} R_{A_\gamma} \\
& = 2 J_{A_\gamma} T_{A_\gamma, B_\gamma} - T_{A_\gamma, B_\gamma} - T_{ B_\gamma, A_\gamma} R_{A_\gamma} \\
& = 2 J_{A_\gamma} T_{A_\gamma, B_\gamma} - T_{A_\gamma, B_\gamma} - \big( \Id + 2 \lambda J_{A_\gamma} R_{B_\gamma} - 2 \lambda J_{B_\gamma}\big) R_{A_\gamma} \\
& = 2 J_{A_\gamma} T_{A_\gamma, B_\gamma} - T_{A_\gamma, B_\gamma} - R_{A_\gamma } - 2 \lambda \big( J_{A_\gamma} R_{B_\gamma} R_{A_\gamma} - J_{B_\gamma } R_{A_\gamma}\big) \\
& = 2 J_{A_\gamma} T_{A_\gamma, B_\gamma} - \big( \Id + 2 \lambda J_{B_\gamma } R_{A_\gamma} - 2 \lambda J_{A_\gamma }\big) - R_{A_\gamma}\nonumber
\\
&\qquad  - 2 \lambda \big( J_{A_\gamma} R_{B_\gamma} R_{A_\gamma} - J_{B_\gamma } R_{A_\gamma}\big) \\
& = 2 J_{A_\gamma} T_{A_\gamma, B_\gamma} - \Id + 2 \lambda J_{A_\gamma} - R_{A_\gamma} - 2 \lambda J_{A_\gamma } R_{B_\gamma } R_{A_\gamma} \\
& = 2 J_{A_\gamma} T_{A_\gamma, B_\gamma} - 2 (1 - \lambda ) J_{A_\gamma} - 2 \lambda J_{A_\gamma} R_{B_\gamma } R_{A_\gamma},
\end{align*}
which verifies \cref{eq:m8}. Next, by using \cref{eq:ros} and \cref{eq:m8}:
\begin{align*} 
R_{A_\gamma} T_{A_\gamma, B_\gamma} - T_{B_\gamma, A_\gamma} R_{A_\gamma} &  = 2 J_{A_\gamma} T_{A_\gamma, B_\gamma} - 2 (1 - \lambda ) J_{A_\gamma} - 2 \lambda J_{A_\gamma} R_{B_\gamma } R_{A_\gamma} \\
& = 2 \big( \gamma J_{A}+ (1 - \gamma ) w \big) T_{A_\gamma, B_\gamma} - 2 (1 - \lambda) \big( \gamma J_{A} + (1 - \gamma ) w \big)\nonumber
\\
&\qquad - 2 \lambda \big( \gamma J_{A} + (1 - \gamma ) w \big) R_{B_\gamma } R_{A_\gamma } \\
& = 2 \gamma J_{A} T_{A_\gamma, B_\gamma } - 2 \gamma (1 - \lambda) J_{A} - 2 \lambda \gamma J_{A} R_{B_\gamma } R_{A_\gamma}.
\end{align*}	
\end{proof}

\begin{theorem}\label{Lem:rs3}\normalfont  Let $\gamma \in \left] 0,1 \right[$, $\lambda \in \left] 0, 1 \right]$ and suppose that $A$ is an affine realtion. Then 
\begin{equation}\label{eq:l1} 
 R_{A_\gamma} T^{n}_{A_\gamma, B_\gamma } = T^{n}_{B_\gamma, A_\gamma } R_{A_\gamma}.  
\end{equation}	
\end{theorem}

\begin{proof}
We will prove by induction that $R_{A_\gamma} T^{n}_{A_\gamma, B_\gamma} = T^{n}_{B_\gamma, A_\gamma } R_{A_\gamma}$. Let $n = 1$. By using \cref{eq:m9}, we obtain 
\begin{align*}
R_{A_\gamma} T_{A_\gamma, B_\gamma} - T_{B_\gamma, A_\gamma} R_{A_\gamma} & = 2 \gamma J_{A} T_{A_\gamma, B_\gamma} - 2 \gamma  (1 - \lambda) J_{A} - 2 \lambda \gamma J_{A} R_{B_\gamma} R_{A_\gamma} \\
& =  J_{A} \Big( 2 \gamma \big( T_{A_\gamma, B_\gamma} - (1 - \lambda) \Id \big) \Big) - 2 \lambda \gamma  J_{A} R_{B_\gamma}R_{A_\gamma} \\
& = J_{A} \Big( 2 \gamma \big( (1 - \lambda ) \Id + \lambda R_{B_\gamma} R_{A_\gamma }\big) - 2 \gamma (1 - \lambda) \Id \Big)  - 2 \lambda \gamma J_{A} R_{B_\gamma}R_{A_\gamma} \\
& =  2 \lambda \gamma J_{A} R_{B_\gamma } R_{A_\gamma } - 2 \lambda \gamma J_{A} R_{B_\gamma} R_{A_\gamma} = 0.
\end{align*}
Hypothesis assumption: when $n = k$;
\begin{equation}\label{eqq:l2}
R_{A_\gamma} T^{k}_{A_\gamma, B_\gamma} -T^{k}_{B_\gamma, A_\gamma} R_{A_\gamma} = 0.
\end{equation}
When $n = k+ 1$ and by using \cref{eqq:l2}, we have 
\begin{align*}
R_{A_\gamma} T^{k+1}_{A_\gamma, B_\gamma} - T^{k+1}_{B_\gamma, A_\gamma} R_{A_\gamma} & = R_{A_\gamma} T^{k}_{A_\gamma, B_\gamma} T_{A_\gamma, B_\gamma} - T^{k}_{B_\gamma, A_\gamma} T_{B_\gamma, A_\gamma} R_{A_\gamma} \\
& = R_{A_\gamma} T^{k}_{A_\gamma, B_\gamma} T_{A_\gamma, B_\gamma} - T^{k}_{B_\gamma, A_\gamma} R_{A_\gamma} T_{A_\gamma, B_\gamma} \\
& = R_{A_\gamma} T^{k}_{A_\gamma, B_\gamma} T_{A_\gamma, B_\gamma} - R_{A_\gamma }  T^{k}_{A_\gamma, B_\gamma} T_{A_\gamma, B_\gamma} \\
& = 0.
\end{align*}
Hence, \eqref{eq:l1} is verified. 
\end{proof}

\begin{lemma}\normalfont Suppose both $A$ and $B$ are affine relation. Then 
\begin{enumerate}
\item\label{rr0} $T_{A_\gamma, B_\gamma}$ and $T_{B_\gamma, A_\gamma}$ are affine.
\item\label{rr1} $T_{A_\gamma, B_\gamma} R_{B_\gamma} R_{A_\gamma} = R_{B_\gamma} R_{A_\gamma} T_{A_\gamma, B_\gamma}$.
\item\label{rr2} $\lambda^{-2} \big( T_{A_\gamma, B_\gamma}T_{B_\gamma, A_\gamma} - T_{B_\gamma, A_\gamma} T_{A_\gamma, B_\gamma}\big)   =  R_{B_\gamma} R^2_{A_\gamma} R_{B_\gamma} - R_{A_\gamma} R^2_{B_\gamma} R_{A_\gamma}$.
\item\label{rr3} $ T_{A_\gamma, B_\gamma}T_{B_\gamma, A_\gamma} = T_{B_\gamma, A_\gamma}T_{A_\gamma, B_\gamma} \iff R_{B_\gamma} R^2_{A_\gamma} R_{B_\gamma} = R_{A_\gamma} R^2_{B_\gamma} R_{A_\gamma} $.
\item\label{rr4} If  $ R^2_{A_\gamma} =  R^2_{B_\gamma}$, then $T_{A_\gamma, B_\gamma}T_{B_\gamma, A_\gamma} = T_{B_\gamma, A_\gamma}T_{A_\gamma, B_\gamma}$.
\end{enumerate}	
\end{lemma}
\begin{proof}
\ref{rr0}: Clear.
\ref{rr1}: From \ref{rr0} we obtain 
\begin{align*}
T_{A_\gamma, B_\gamma} R_{B_\gamma}R_{A_\gamma} & = T_{A_\gamma, B_\gamma} \big(  \lambda^{-1} T_{A_\gamma, B_\gamma} - \lambda^{-1}(1 - \lambda)\Id \big) \\
& = \lambda^{-1} T^{2}_{(A_\gamma, B_\gamma)} - \lambda^{-1}(1 - \lambda)T_{A_\gamma, B_\gamma} \\
& = \big( \lambda^{-1} T_{A_\gamma, B_\gamma} - \lambda^{-1}(1 - \lambda)  \Id \big) T_{A_\gamma, B_\gamma} \\
& = R_{B_\gamma} R_{A_\gamma} T_{A_\gamma, B_\gamma}.	
\end{align*}
\ref{rr2}: By using \eqref{twithouttide}, we have 
\begin{align*}
\lambda^{-2} \big( T_{A_\gamma, B_\gamma}  T_{B_\gamma, A_\gamma} \big) & = \lambda^{-2} \big( (1 - \lambda) \Id + \lambda R_{B_\gamma} R_{A_\gamma} \big) \big( (1 - \lambda) \Id + \lambda R_{A_\gamma} R_{B_\gamma} \big)  	
\end{align*}
Hence,
\begin{align}
\lambda^{-2} \big( T_{A_\gamma, B_\gamma}  T_{B_\gamma, A_\gamma} \big) & = \lambda^{-2} \big( (1 - \lambda)^2 \Id + \lambda ( 1 - \lambda) R_{A_\gamma} R_{B_\gamma} + \lambda (1 - \lambda) R_{B_\gamma} R_{A_\gamma} \nonumber
\\
&\qquad+ \lambda^2 R_{B_\gamma} R^2_{A_\gamma} R_{B_\gamma}\big).\label{nn1}
\end{align} 
Moreover, 
\begin{align*}
	\lambda^{-2} \big( T_{B_\gamma, A_\gamma}  T_{A_\gamma, B_\gamma} \big) & = \lambda^{-2} \big( (1 - \lambda) \Id + \lambda R_{A_\gamma} R_{B_\gamma} \big) \big( (1 - \lambda) \Id + \lambda R_{B_\gamma} R_{A_\gamma} \big)  	
\end{align*}
Hence,
\begin{align}
	\lambda^{-2} \big( T_{B_\gamma, A_\gamma}  T_{A_\gamma, B_\gamma} \big) & = \lambda^{-2} \big( (1 - \lambda)^2 \Id + \lambda ( 1 - \lambda) R_{B_\gamma} R_{A_\gamma} + \lambda (1 - \lambda) R_{A_\gamma} R_{B_\gamma} \nonumber
	\\
	&\qquad+ \lambda^2 R_{A_\gamma} R^2_{B_\gamma} R_{A_\gamma}\big).\label{nn2}
\end{align}
Subtracting \eqref{nn2} from \eqref{nn1} gives
$$ \lambda^{-2} \big( T_{A_\gamma, B_\gamma}T_{B_\gamma, A_\gamma} - T_{B_\gamma, A_\gamma} T_{A_\gamma, B_\gamma}\big)   =  R_{B_\gamma} R^2_{A_\gamma} R_{B_\gamma} - R_{A_\gamma} R^2_{B_\gamma} R_{A_\gamma}. $$
\ref{rr3} and \ref{rr4}: They follow from \ref{rr2}. 
\end{proof}

\begin{proposition}\label{EEXA1} Let $U $ be closed linear subspace and $A = \Id + v$, where $v \in U^\perp$. Let $ B = \PR_{a+ U}$, where $a \in U^\perp$ and $a \neq v$. Then the following hold:
\begin{enumerate}
\item\label{Ex.11}  We have 
\begin{align*}
T_{A_\gamma, B_\gamma} (x) & = \big( 1  -  \lambda \gamma (3 - 2 \gamma) \big) x +  \lambda \gamma ( 1 -  \gamma ) \PR_{ U} x + k,
\end{align*}
where $$ k =  \lambda \gamma \big( (2 \gamma - 1)v + 4 (1 - \gamma ) w -2 ( 1 - \gamma ) \PR_{ U} w - 2 a \big).$$
\item \label{Ex.12}  We  have 
\begin{align*}
T_{B_\gamma, A_\gamma } (x) & = \big( 1 - \lambda \gamma (3 - 2 \gamma )\big) x + \lambda \gamma (1 - \gamma ) \PR_{ U} x + l,
\end{align*}
where 
\begin{align*}
 l & = \lambda \gamma \big( 2 (1 - \gamma ) a + v + 2 ( 1 - \gamma ) w \big).
\end{align*}
\item\label{Ex.13}  We have 
\begin{align*}
 R_{A_\gamma} T_{A_\gamma, B_\gamma} (x) &=  T_{B_\gamma, A_\gamma}  R_{A_\gamma}  (x)\\
	& = \big( 1 - \gamma  \big) \Big(  \big(  \lambda \gamma \big(  3 - 2 \gamma \big) -1 \big) x - \lambda \gamma \big(  1 - \gamma \big)  \PR_{U} x\Big) + h,
\end{align*}
where 
\begin{align*}
h & =  \gamma \big[  \lambda \gamma \big(2 \gamma - 3 \big) + 1 + \lambda \big] v + 2 \big(   1 - \gamma \big) \big[  \big(  1 - 2 \lambda \gamma  \big(1 - \gamma \big) \big)  w \\
& \ \ \ \ \ \ +   \lambda \gamma   \big(  1 - \gamma \big) \PR_{U}  w \big] + 2 \lambda \gamma \big(  1 - \gamma \big) a.
\end{align*}

\item\label{Ex.14} We have 
\begin{align*}
R_{B_\gamma} T_{A_\gamma, B_\gamma} & = \big(  1 - 2 \gamma \big) \Big(    \big(  \lambda  \gamma ( 3 - 2 \gamma ) - 1\big) x - \lambda \gamma ( 1 - \gamma ) \PR_{U} x \Big) + m,
\end{align*}
where 
\begin{align*}
m & = \big(  1 - 2 \gamma \big) \lambda \gamma \big[  \big( 1 - 2 \gamma  \big)  v + 2 \big(   1 - \gamma \big)  \PR_{U} w - 4 \big(  1 - \gamma \big) w \big] \\
& \ \ \ \ \ \ + 2 \big[  \big( 1 - \gamma   \big)  w + \gamma \big(  \lambda- 1 - 2 \lambda \gamma \big) a \big].
\end{align*}

\item \label{Ex.15} We have 
\begin{align*}
 T_{B_\gamma, A_\gamma} R_{B_\gamma} & =  \big(  1 - 2 \gamma \big) \Big(    \big(  \lambda  \gamma ( 3 - 2 \gamma ) - 1\big) x - \gamma \big(  1 + \lambda - \lambda \gamma \big( 5- 3 \gamma  \big)  \big) \PR_{U} x + s,
\end{align*}
where 
\begin{align*}
s & = \lambda \gamma v + 2 \Big(  \lambda \gamma^2 - \big( 2 \lambda + 1 \big) \gamma +1  \Big) w \\
& \ \ \ \ \ \ - 2 \gamma \Big(   \lambda \big(  3 \gamma + 4 \big) + 1\Big) a + 2 \lambda \gamma \big( 1 - \gamma  \big)^2 \PR_{U}w.
\end{align*}
\item\label{Ex.16} 
$R_{B_\gamma} T_{A_\gamma, B_\gamma} \neq T_{B_\gamma, A_\gamma} R_{B_\gamma} $.
\item\label{Ex.17}  We have 
\begin{align*}
R_{B_\gamma} T_{B_\gamma, A_\gamma} & = \big(  2 \gamma - 1 \big)  \Big(  1 - \lambda \gamma \big(  3 - 2 \gamma \big) \Big) x + \gamma \Big(   \lambda \gamma \big(  5 - 3 \gamma \big)  - \lambda -1 \Big)  \PR_{U} x  + b,
\end{align*}
where 
\begin{align*}
b & =- 2 \gamma \Big(   \lambda \gamma \big( 2 \gamma - 3 \big) + \lambda + 1 \Big)  a - 2 \Big(  \gamma \Big(   \lambda \gamma \big( 2 \gamma - 3  \big) + 1 \Big) - 1\Big) w + \lambda \gamma \big(   2 \gamma -1 \big) v \\
&  \ \ \ \ \  - 2 \lambda \gamma^2 \big( 1 - \gamma  \big) \PR_{U}  w.
\end{align*}
\item\label{Ex.18}  We have 
\begin{align*}
T_{A_\gamma, B_\gamma} R_{B_\gamma} & = \big(  2 \gamma - 1 \big)  \Big(  1 - \lambda \gamma \big(  3 - 2 \gamma \big) \Big) x + \gamma \Big(   \lambda \gamma \big(  5 - 3 \gamma \big)  - \lambda -1 \Big)  \PR_{U} x  + c,
\end{align*}
where 
\begin{align*}
c & =  - 2 \gamma \Big(   \lambda \gamma \big( 2 \gamma - 3 \big) + \lambda + 1 \Big)  a - 2 \Big(  \gamma \Big(   \lambda \gamma \big( 2 \gamma - 3  \big) + \lambda + 1 \Big) - 1\Big) w + \lambda \gamma \big(   2 \gamma -1 \big) v \\
&  \ \ \ \ \ \  - 2 \lambda \gamma^2 \big( 1 - \gamma  \big) \PR_{U}   w .
\end{align*}
\item\label{Ex.19}  $R_{B_\gamma} T_{B_\gamma, A_\gamma} \neq T_{A_\gamma, B_\gamma} R_{B_\gamma}$.
\end{enumerate}
\end{proposition}

\begin{proof}
\ref{Ex.11}:  It follows from  \cref{rea1} \ref{tab:1}. 
\ref{Ex.12}	:  It follows from  \cref{rea1} \ref{tab:2}.
\ref{Ex.13}	:  By using \ref{Ex.11}, \ref{Ex.12}, and \cref{rea1}\ref{eqq:4} we have 
\begin{align*}
 R_{A_\gamma} T_{A_\gamma, B_\gamma} (x) &= \Big(  \gamma v - \big( 1 - \gamma  \big) \Id + 2 \big( 1 - \gamma  \big) w\Big) T_{A_\gamma, B_\gamma} (x) \\
 & = \gamma v + 2 \big( 1 - \gamma  \big) w - \big( 1 - \gamma \big) T_{A_\gamma, B_\gamma} (x) \\
 & = \gamma v - \big( 1 - \gamma  \big) \lambda \gamma \big(   2 \gamma -1\big) v + 2 \big( 1 - \gamma \big) w - 4 \lambda \gamma  \big( 1- \gamma  \big)^2 w \\
 &  \ \ \  - \big( 1 - \gamma  \big) \Big(  \Big(  1 - \lambda \gamma \big(  3 - 2 \gamma \big) \Big) x + \lambda \gamma \big(  1 - \gamma \big) \PR_{U} x  + \lambda \gamma \Big(  -2 \big(  1 - \gamma \big) \PR_{U}  w - 2 a\Big)\Big) \\
 & = \big( 1 - \gamma  \big) \Big(  \big(  \lambda \gamma \big(  3 - 2 \gamma \big) -1 \big) x - \lambda \gamma \big(  1 - \gamma \big)  \PR_{U} x\Big) \\
 &  \ \ \ \ \ \ + \gamma \big[  \lambda \gamma \big(2 \gamma - 3 \big) + 1 + \lambda \big] v + 2 \big(   1 - \gamma \big) \big[  \big(  1 - 2 \lambda \gamma  \big(1 - \gamma \big) \big)  w \\
& \ \ \ \ \ \ +   \lambda \gamma   \big(  1 - \gamma \big) \PR_{U}  w \big] + 2 \lambda \gamma \big(  1 - \gamma \big) a.
\end{align*}
Moreover, 
\begin{align*}
T_{B_\gamma, A_\gamma}  R_{A_\gamma}  (x) &=  \big(   1 - \lambda \gamma \big(  3 - 2 \gamma \big) \big) R_{A_\gamma} (x)  + \lambda \gamma \big(   1 - \gamma \big) \PR_{U}  R_{A_\gamma } (x)  \\
& \ \ \ \ \ \ +   \lambda \gamma   \big(  2 \big( 1 - \gamma \big) a + v+ 2 \big(  1 - \gamma \big) w \big) . \\
 & = \big( 1 - \gamma  \big) \Big(  \big(  \lambda \gamma \big(  3 - 2 \gamma \big) -1 \big) x - \lambda \gamma \big(  1 - \gamma \big)  \PR_{U} x\Big) \\
 &  \ \ \ \ \ \ + \gamma \big[  \lambda \gamma \big(2 \gamma - 3 \big) + 1 + \lambda \big] v + 2 \big(   1 - \gamma \big) \big[  \big(  1 - 2 \lambda \gamma  \big(1 - \gamma \big) \big)  w \\
& \ \ \ \ \ \ +   \lambda \gamma   \big(  1 - \gamma \big) \PR_{U}  w \big] + 2 \lambda \gamma \big(  1 - \gamma \big) a.
\end{align*}
Therefore, 
$$    R_{A_\gamma} T_{A_\gamma, B_\gamma} (x)= T_{B_\gamma, A_\gamma}  R_{A_\gamma}  (x) .$$
and \ref{Ex.13} is verified. \\
\ref{Ex.14}: By using \cref{rea1}\ref{eqq:2} and \ref{Ex.11}  we have 
\begin{align*}
R_{B_\gamma} T_{A_\gamma, B_\gamma} (x) & = \Big(   \big( 2 \gamma - 1\big) \Id - \gamma \PR_{U} - 2 \gamma a + 2 \big( 1 - \gamma \big) w\Big) T_{A_\gamma, B_\gamma} (x)\\
& = 2 \big( 1 - \gamma \big) w - 2 \gamma a - \big( 1  - 2 \gamma \big) T_{A_\gamma, B_\gamma} (x) - \gamma \PR_{U} T_{A_\gamma, B_\gamma} (x)   \\
& =  2 \big( 1 - \gamma \big) w - 2 \gamma a - \big( 1  - 2 \gamma \big) \Big[   \Big(  1- \lambda \gamma \big( 3-2 \gamma \big)  \Big) x + \lambda \gamma \big(  1 - \gamma  \big) \PR_{U} x + k\Big] \\
&= \big(  1 - 2 \gamma \big) \Big[   \Big(   \lambda \gamma \big( 3 - 2 \gamma  \big) - 1 \Big) x - \lambda \gamma \big( 1 - \gamma  \big) \PR_{U} x \Big] \\
 &  \ \ \ \ \ \ + \lambda \gamma  \big(  1 - 2 \gamma \big) \Big[   \big( 1 -2 \gamma \big) v + 2 \big( 1 - \gamma  \big) \PR_{U} w - 4 \big(  1 - \gamma \big) w \Big] \\
 &  \ \ \ \ \ \ + 2 \Big[  \big( 1 - \gamma  \big)  w + \gamma \big(  \lambda-1-2 \lambda \gamma \big) a \Big] .
\end{align*}

\ref{Ex.15}: By using \ref{Ex.12} and \cref{rea1}\ref{eqq:2}  we have
\begin{align*}
 T_{B_\gamma, A_\gamma} R_{B_\gamma} & = \Big(    \Big(  1 - \lambda \gamma \big(   3 - 2 \gamma \big)   \Id \Big)  + \lambda \gamma \big( 1 - \gamma  \big) \PR_{U} + l \Big) R_{B_\gamma } (x) \\
 & =  \Big(  1 - \lambda \gamma \big(   3 - 2 \gamma \big)  R_{B_\gamma } (x) +  \lambda \gamma \big( 1 - \gamma  \big) \PR_{U}   \big( R_{B_\gamma } (x)  \big) + l \big(  R_{B_\gamma } (x)  \big) \\
 & = \Big(  1 - \lambda \gamma \big(  3 - 2 \gamma \big) \Big) \big(  2 \gamma - 1\big) x - \gamma \Big(  1 + \lambda - \lambda \gamma \big( 5 - 3 \gamma  \big) \Big) \PR_{U} x \\
 &  \ \ \ \ \ \ + 2 \lambda \gamma  \big( 1 - \gamma \big)  \big( a + w \big) + \Big(  1 - \lambda \gamma \big(  3 - 2 \gamma  \big) \Big)  \Big(  2 \big(  1 - \gamma \big) w - 2 \gamma a \Big) \\
 &  \ \ \ \ \ \  + 2 \lambda \gamma \big(  1 - \gamma \big)^2 \PR_{U} w + \lambda \gamma v \\
 & =  \big(  1 - 2 \gamma \big) \Big(    \big(  \lambda  \gamma ( 3 - 2 \gamma ) - 1\big) x - \gamma \big(  1 + \lambda - \lambda \gamma \big( 5- 3 \gamma  \big)  \big) \PR_{U} x \\
  &  \ \ \ \ \ \  + \lambda \gamma v + 2 \Big(  \lambda \gamma^2 - \big( 2 \lambda + 1 \big) \gamma +1  \Big) w \\
& \ \ \ \ \ \ - 2 \gamma \Big(   \lambda \big(  3 \gamma + 4 \big) + 1\Big) a + 2 \lambda \gamma \big( 1 - \gamma  \big)^2 \PR_{U}w.
\end{align*}
\ref{Ex.16}: It follows from \ref{Ex.14} and \ref{Ex.15}. \\
\ref{Ex.17}:  By using \cref{rea1}\ref{eqq:2} and \ref{Ex.12} we have
\begin{align*}   
R_{B_\gamma } T_{B_\gamma, A_\gamma } (x) & = \Big(   \big(  2 \gamma -1 \big) \Id - \gamma \PR_{U} -2 \gamma a + 2 \big(  1- \gamma \big) w \Big) T_{B_\gamma, A_\gamma } (x)\\
& =  \big(  2 \gamma -1 \big) T_{B_\gamma, A_\gamma } (x) - \gamma \PR_{U} \big(  T_{B_\gamma, A_\gamma } (x) \big) -2 \gamma a + 2 \big( 1 - \gamma \big) w \\
& = \big(  2 \gamma -1 \big) \Big[    \Big(   1 - \lambda \gamma \big( 3 - 2 \gamma  \big)  \Big) x + \lambda \gamma \big( 1 - \gamma  \big) \PR_{U} x + l \Big] \\
& \ \ \ \ \ \  - \gamma \PR_{U} \Big[    \Big(   1 - \lambda \gamma \big( 3 - 2 \gamma  \big)  \Big) x + \lambda \gamma \big( 1 - \gamma  \big) \PR_{U} x + l \Big]  - 2 \gamma a + 2 \big(  1- \gamma \big) w \\
&= \big(  2 \gamma -1 \big) \Big[  \Big(   1- \lambda \gamma \big(  3 - 2 \gamma \big)\Big) x + \lambda \gamma \big( 1- \gamma  \big) \PR_{U} x  \Big] \\
& \ \ \ \ \ \  - \gamma \Big[  \Big(  1 -  \lambda \gamma \big(  3 - 2 \gamma  \big)\Big)  \PR_{U}  x + \lambda \gamma \big(  1- \gamma \big) \PR_{U}  x \Big] \\
& \ \ \ \ \ \  + \big( 2 \gamma -1  \big) l - \gamma \PR_{U} l - 2 \gamma a + 2 \big(  1 - \gamma \big) w \\
& = \big(  2 \gamma - 1 \big)  \Big(  1 - \lambda \gamma \big(  3 - 2 \gamma \big) \Big) x + \gamma \Big(   \lambda \gamma \big(  5 - 3 \gamma \big)  - \lambda -1 \Big)  \PR_{U} x  \\
&  \ \ \ \ \ \  - 2 \gamma \Big(   \lambda \gamma \big( 2 \gamma - 3 \big) + \lambda + 1 \Big)  a - 2 \Big(  \gamma \Big(   \lambda \gamma \big( 2 \gamma - 3  \big) + 1 \Big) - 1\Big) w + \lambda \gamma \big(   2 \gamma -1 \big) v \\
&  \ \ \ \ \ \  - 2 \lambda \gamma^2 \big( 1 - \gamma  \big) \PR_{U}  w.
\end{align*}
\ref{Ex.18}:   By using \cref{rea1}\ref{eqq:2} and \ref{Ex.12} we have
\begin{align*} 
T_{A_\gamma, B_\gamma } R_{B_\gamma } (x) & = \Big(  1 - \lambda \gamma \big(   3 - 2 \gamma \big)  \Id + \lambda \gamma  \big(  1 - \gamma \big) \PR_{U} + k \Big)  R_{B_\gamma } (x) \\
& = \Big( 1 - \lambda \gamma \big(  3 - 2 \gamma \big) \Big) R_{B_\gamma } (x) + \lambda \gamma \big(  1 - \gamma \big) \PR_{U} R_{B_\gamma } (x) + k \\
& = \Big(  1 - \lambda \gamma \big( 3 - 2 \gamma  \big) \Big) \Big(   \big(  2 \gamma -1  \big) x - \gamma \PR_{U} x \Big) +  \Big(  1 - \lambda \gamma \big( 3 - 2 \gamma  \big) \Big) \Big(  2 \big(  1 - \gamma \big) w - 2 \gamma a \Big) \\
&  \ \ \ \ \ \ + \lambda \gamma \big( 1 - \gamma \big) \PR_{U} \Big(  \big( 2 \gamma -1  \big)  x - \gamma \PR_{U} x\Big) + \lambda \gamma \big( 1 - \gamma \big) \PR_{U} \Big( 2 \big(  1- \gamma \big)  w - 2 \gamma a \Big) + k \\
& = \big(  2 \gamma - 1 \big)  \Big(  1 - \lambda \gamma \big(  3 - 2 \gamma \big) \Big) x + \gamma \Big(   \lambda \gamma \big(  5 - 3 \gamma \big)  - \lambda -1 \Big)  \PR_{U} x  \\
&  \ \ \ \ \ \  - 2 \gamma \Big(   \lambda \gamma \big( 2 \gamma - 3 \big) + \lambda + 1 \Big)  a - 2 \Big(  \gamma \Big(   \lambda \gamma \big( 2 \gamma - 3  \big) + \lambda + 1 \Big) - 1\Big) w  \\
&  \ \ \ \ \ \  + \lambda \gamma \big(   2 \gamma -1 \big) v- 2 \lambda \gamma^2 \big( 1 - \gamma  \big) \PR_{U}   w .
\end{align*}
\ref{Ex.19}: It follows from \ref{Ex.17} and \ref{Ex.18}. 
\end{proof}



\end{document}